\documentclass[12pt]{amsart}
\usepackage[margin=1in]{geometry}
\usepackage{graphicx}
\usepackage{amssymb}
\usepackage{amsmath}
\usepackage{amsfonts}
\usepackage{amstext}
\usepackage{amsthm}
\usepackage{setspace}
\usepackage{array}
\usepackage{textcomp}
\usepackage{tipa}
\usepackage{graphicx}

\usepackage{wrapfig}
\usepackage{caption}
\usepackage{subcaption}

\newcommand{\be}{\begin{equation} }
\newcommand{\ee}{\end{equation} }
\newcommand{\bee}{\begin{eqnarray} }
\newcommand{\eee}{\end{eqnarray} }

\renewcommand{\Re}{{\operatorname{Re}\,}}
\renewcommand{\Im}{{\operatorname{Im}\,}}

\newcommand{\bbb}{|\!|\!|}

\numberwithin{equation}{section}

\newtheorem{Lem}{Lemma}

\newtheorem{Thm}{Theorem}

\newtheorem{definition}{Definition}
\newtheorem{remark}{Remark}

\newtheorem{proposition}{Proposition}

\newcommand{\set}[1]{\ensuremath{\{#1\}}}

\newcommand{\E}[1]{\ensuremath{\mathbb E\left[ #1 \right]}}
\newcommand{\EE}[1]{\mathbb E}

\newcommand{\curly}[1]{\ensuremath{\mathcal #1}}

\newcommand{\ep}{\ensuremath{\epsilon}}
\newcommand{\R}{\ensuremath{\mathbb R}}

\newcommand{\w}{\ensuremath{\omega}}
\newcommand{\W}{\ensuremath{\Omega}}

\newcommand{\twiddle}[1]{\ensuremath{\widetilde{#1}}}

\newcommand{\case}[2]{\ensuremath{#1, \text{~if~} #2}}

\newcommand{\setst}[2]{\ensuremath{\left\{#1\,\middle|\,#2\right\}}}
\newcommand{\intersection}{\ensuremath{\cap}}

\newcommand{\abs}[1]{\left\lvert #1 \right\rvert}

\newcommand{\inprod}[2]{\ensuremath{\left\langle#1,#2\right\rangle}}

\newcommand{\gives}{\ensuremath{\rightarrow}}

\newcommand{\x}{\ensuremath{\times}}

\newcommand{\lr}[1]{\ensuremath{\left(#1\right)}}
\renewcommand{\Re}{\ensuremath{\mathrm{Re} \ }}
\renewcommand{\Im}{\ensuremath{\mathrm{Im} \ }}
\newcommand{\dell}{\ensuremath{\partial}}

\DeclareMathOperator{\Cov}{Cov}

\DeclareMathOperator{\sgn}{sgn}
\DeclareMathOperator{\Den}{Den}
\newcommand{\twomat}[4]{\ensuremath{ \left(\begin{array}{cc} #1 & #2 \\
#3 & #4 \end{array}\right)}}

\newcommand{\half}{{\frac{1}{2}}}
\newcommand{\Z}{{\mathbb Z}}
\newcommand{\ihbar}{{\frac{i}{h}}}
\newcommand{\ncal}{{\mathcal N}}
\newcommand{\acal}{{\mathcal A}}
\newcommand{\fcal}{{\mathcal F}}
\newcommand{\hcal}{{\mathcal H}}
\DeclareMathOperator{\supp}{supp}

\newcommand{\ZN}{|Z_{\Phi_N}|}

\title{Nodal Sets of Random Eigenfunctions for the Isotropic Harmonic
Oscillator}            
\begin{document}
\author{Boris Hanin, Steve Zelditch, Peng Zhou}

\address{Department of Mathematics, Northwestern  University, Evanston, IL
60208, USA}
\email[B. Hanin]{bhanin@math.northwestern.edu}
\email[S. Zelditch]{zelditch@math.northwestern.edu}
\email[P. Zhou]{pengzhou@math.northwestern.edu}

\thanks{Research partially supported by NSF grant DMS-1206527  .}
\maketitle
\setcounter{section}{-1}

\begin{abstract}  The expected hypersurface measure $\hcal^{d-1}(Z_{E,h}
\cap B(r, x))$ of nodal sets of random
eigenfunctions of eigenvalue $E$ of the semi-classical isotropic harmonic
oscillator in balls $B(r, x) \subset
\R^d$ is determined as $h\gives 0.$ In the allowed region the volumes are
of order $h^{-1},$ while in the
forbidden region they are of order $h^{-\half}$.
\end{abstract}
\section{Introduction}

This article is concerned with the semi-classical asymptotics of nodal (i.e.
zero) sets of random eigenfunctions
of the isotropic Harmonic Oscillator,
\begin{equation} \label{Hh}
H_{h} =  \sum_{j = 1}^d \left(- \frac{h^2}{2}   \frac{\partial^2 }{\partial
x_j^2} + \frac{x_j^2}{2} \right),
\end{equation}
on $L^2(\R^d)$. Random isotropic Hermite functions of fixed degree have
an $SO(d-1)$ symmetry and are in
some ways analogous to random spherical harmonics of fixed degree on
$L^2(S^{d})$, whose nodal sets have been the
subject of many recent studies (see e.g. \cite{NS}).

\begin{wrapfigure}{r}{0.4\textwidth}
\vspace{-11pt}
\begin{center}
  \includegraphics[width=1.0\linewidth]{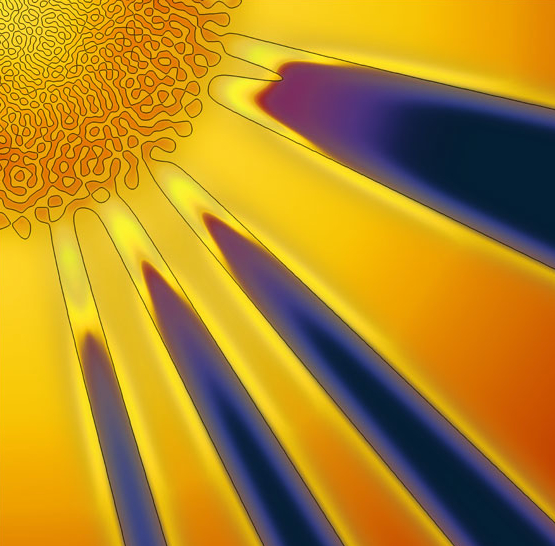}
\end{center}
\caption{}
\vspace{-10pt}
\end{wrapfigure}

However, there is a fundamentally new aspect to eigenfunctions of
Schr\"odinger operators on $\R^d$, namely the
existence of allowed and forbidden regions. In the allowed region, $H_{h}$
behaves like an elliptic operator
(with parameter $h$) and the nodal sets of eigenfunctions behave similarly
to those of eigenfunctions of the
Laplace operator on a Riemannian manifold.  For instance the classical
estimates of Donnelly-Fefferman \cite{DF}
for the hypersurface measure of nodal sets of Laplace eigenfunctions for
real analytic metrics has an analogue
for semiclassical eigenfunctions of Schr\"odinger operators with analytic
metrics and potentials (see Long Jin
\cite{J}).  In the $C^{\infty}$ case one has lower bounds on hypersurface
volumes of nodal sets in the allowed
region which are similar to those for smooth metrics \cite{ZZ}.  However,
there do not seem to exist prior
results on nodal volumes in the forbidden region, although there do exist
numerical and heuristic results on
random eigenfunctions of \eqref{Hh} in Bies-Heller \cite{BH}. In the
forbidden region, eigenfunctions are
exponentially decaying and it is not clear to what extent they oscillate and
have zeros; in dimension one,
eigenfunctions of the Harmonic oscillator have no zeros in the forbidden
region. To gain insight into the
behavior of nodal sets in the forbidden region, we randomize the problem and
consider Gaussian random
eigenfunctions of \eqref{Hh}. Our main results show that the expected
hypersurface measure of nodal sets in
compact subsets of the allowed region are of order $h^{-1}$, parallel to
that of Laplace eigenfunctions, while in
the forbidden region they are of order $h^{-\half}$.

To state our main result, Theorem \ref{T:Main}, we introduce some notation
and background. Acting on
$L^2(\R^d,dx),$ $H_h$ has an orthonormal basis of eigenfunctions
\begin{equation}
\phi_{\alpha,h}(x)=h^{-d/4}p_{\alpha}\lr{x\cdot
h^{-1/2}}e^{-x^2/2h},\label{E:Scaling Relation}
\end{equation}
where $\alpha=\lr{\alpha_1,\ldots, \alpha_d}\geq (0,\ldots,0)$ is a
$d-$dimensional multi-index and
$p_{\alpha}(x)$ is the product $\prod_{j=1}^d (2^{\alpha_j} \alpha_j!)^{-1/2}
\pi^{-1/4}p_{\alpha_j}(x_j)$ of the
hermite polynomials $p_k$ (of degree $k$)
in one variable. The eigenvalue of $\phi_{\alpha,h}$ is given by
\begin{equation} \label{EV}
H_{h} \phi_{\alpha,h} = h (|\alpha|+d/2) \phi_{\alpha,h}.
\end{equation}
The multiplicity of the eigenvalue $ h (|\alpha|+d/2)$ is the partition
function of $|\alpha|$, i.e. the number
of $\alpha=\lr{\alpha_1,\ldots, \alpha_d}\geq (0,\ldots,0)$  with a fixed
value of $|\alpha|$. The high
multiplicity of the eigenvalues is similar in order of magnitude to that of
the eigenvalues of $\Delta$ on the
standard $S^{d-1}$.

The semi-classical asymptotics of eigenfunctions is the asymptotics as $h
\to 0$ where the energy level $E_h$
satisfies $E_h  \to E$. This corresponds to fixing an energy level of the
classical Hamiltonian
$$H(x, \xi) = \half (|\xi|^2 + |x|^2): T^* \R^m \to \R.$$
We refer to \cite{Zw} for this and other background on semi-classical
asymptotics of Schr\"odinger operators. For
the remainder of this paper we fix  $E>0$ and set
\[h_N:=\frac{E}{N+\frac{d}{2}}.\]
We will usually write $h=h_N.$ We then consider the eigenspace
\begin{equation}\label{VN} V_N = \mbox{Span} \{\phi_{\alpha, h_N}, |\alpha|
= N\}. \end{equation}

\begin{definition}
  A Gaussian random eigenfunction for $H_h$ with eigenvalue $E$ is the
  random series
$$ \Phi_N(x):=\sum_{\abs{\alpha}=N}  a_{\alpha}\phi_{\alpha,h_N}(x),  $$
for $a_{\alpha}\sim N(0,1)_{\R}$ i.i.d. Equivalently, it is the Gaussian
measure $\gamma_N$ on $V_N$
which is given by $e^{- \sum_{\alpha} |a_{\alpha}|^2/2} \prod d
a_{\alpha}$.
\end{definition}
\noindent We denote by $$Z_{\Phi_N} = \{x: \Phi_N(x) = 0\} $$
 the nodal set of $\Phi_N$ and by $\ZN$ the random measure of integration
 over
$Z_{\Phi_N}$ with respect to
 the Euclidean surface measure (the Hausdorff measure) of the nodal set.
 Thus for any ball $B \subset \R^d$,
$$\ZN (B) = \hcal^{d-1} (B \cap Z_{\Phi_N}).$$
Thus $\E \ZN$ is a measure on $\R^n$ given by
$$\E \ZN (B) = \int_{V_N}  \hcal^{d-1} (B \cap Z_{\Phi_N}) d\gamma_N. $$

The allowed region $\acal_E$, resp. the  forbidden region $\fcal_E$ are
defined respectively by \begin{equation}
\label{AF}
\acal_E = \{x:\abs{x}^2<2E\}, \quad \fcal_E = \{x:\abs{x}^2>2E\}.
\end{equation}
Thus, $\acal_E$ is the projection to $\R^d$ of the energy surface $\{H = E\}
\subset T^* \R^d$ and $\fcal_E$ is
its complement. The boundary of $\acal_E$ is the known as the caustic set
and is denoted $\partial \acal_E$ or
$\{|x| = 2 E\}$. Our main result is:

\begin{Thm}\label{T:Main}
Let $x\in \R^d$ such that $0<\abs{x}\neq \sqrt{2E}.$ Then the measure $\E
\ZN$ has a density $F_N(x)$ with
respect to Lebesgue measure given by
$$\left\{\begin{array}{ll}
\mbox{If}~x\in \acal_E\backslash \set{0},  &
F_N(x) \simeq h^{-1}\cdot c_d \sqrt{2E-\abs{x}^2}\lr{1+O(h)}\label{E:Allowed
Density} \\ & \\
\mbox{If}~ x\in \fcal_E, &
F_N(x) \simeq h^{-1/2}\cdot C_d
\frac{E^{1/2}}{\abs{x}^{1/2}\lr{\abs{x}^2-2E}^{1/4}}
\lr{1+O(h)}\label{E:Forbidden Density}
 \end{array}, \right.$$
where the implied constants in the `$O$' symbols are uniform on compact
subsets of the interiors of
$\acal_E\backslash\set{0}$ and $\fcal_E$, and where
\[c_d=
\frac{\Gamma\lr{\frac{d+1}{2}}}{\sqrt{d\pi}\Gamma\lr{\frac{d}{2}}}\qquad
\text{and}\qquad C_d =
\frac{\Gamma\lr{\frac{d}{2}}}{\sqrt{\pi}\Gamma\lr{\frac{d-1}{2}}}.\]
\end{Thm}
The novel aspect of Theorem \ref{T:Main} is the different growth rates in
$h$ for the density of zeros in the
allowed and forbidden region. Let us explain briefly why this happens. As
recalled in Lemma \ref{L:Gaussian KR}
of \S \ref{S:KR}, $F_N(x)$ scales like the square root of the operator norm
of the $d\x d$ matrix
\begin{equation}
\lr{\W_{x,E}}_{1\leq j,k\leq d}
=\frac{\Pi_{h,E}(x,x)\dell_{x_k}\dell_{y_j}|_{x=y}\Pi_{h,E}(x,y)-\dell_{x_k}|_{x=y}\Pi_{h,E}(x,y)\cdot
\dell_{y_j}|_{x=y}\Pi_{h,E}(x,y)}{\Pi_{h,E}(x,x)^2},\label{E:Density
Scaling}
\end{equation}
where $\Pi_{h,E}$ is the spectral projector for $H_h$ onto the eigenspace
with eigenvalue $E.$ Proposition
\ref{P:Cov Matrix Asymptotics} shows that $\W_{x,E}$ is a diagonal matrix
times $h^{-2}$ for  $x\in \curly A_E$
$j=2$ and $h^{-1}$  when $x\in \curly F_E$.

These different powers of $h$ in $\Omega_{x,E}$ come from
Proposition \ref{BIGPROP}, which gives different
oscillatory integral representations for $\Pi_{h,E}(x,y)$ in the allowed and
forbidden regions. Let us  write
them schematically as
\[\Pi_{h,E}(x,y)=\int A(\zeta) e^{\frac{i}{h}S(\zeta, x,y)}d\zeta.\]
The amplitude $A$ is independent of $x,y.$ Differentiating under the
integral, we see that the first term in the
numerator of \eqref{E:Density Scaling}, is
\begin{align}
\label{E:Schematic 0} &\frac{i}{h}\int \int A(\zeta_1)\cdot A(\zeta_2)\left[
\dell_{x_k}\dell_{y_j}|_{x=y}
S(\zeta_1, x,y)\right]
e^{\frac{i}{h}\lr{S(\zeta_1,x,y)+S(\zeta_2,x,y)}}d\zeta_1d\zeta_2\\
\label{E:Schematic 1}& - \frac{1}{h^2}\int A(\zeta_2)
e^{\frac{i}{h}S(\zeta_2,x,x)} d\zeta_2\cdot \int A(\zeta_1)
\dell_{x_k}|_{x=y}S(\zeta_1,x,y) \cdot
\dell_{y_j}S(\zeta_1,x,y)e^{\frac{i}{h}S(\zeta_1,x,y)}d\zeta_1.
\end{align}
The other term, $-\dell_{x_k}|_{x=y}\Pi_{h,E}(x,y)\cdot
\dell_{y_j}|_{x=y}\Pi_{h,E}(x,y),$ in the numerator of
\eqref{E:Density Scaling} is
\begin{equation}
\frac{1}{h^2} \int A(\zeta_1)\dell_{x_k}|_{x=y}S(\zeta_1,x,y)
e^{\frac{i}{h}\lr{S(\zeta_1,x,y)}}d\zeta_1\cdot
\int A(\zeta_2) \dell_{y_j}S(\zeta_2,x,y)
e^{\frac{i}{h}\lr{S(\zeta_2,x,y)}}d\zeta_2.\label{E:Schematic 2}
\end{equation}
By the method of stationary phase, the above integrals localize to the
critical point set of $S.$ In the
forbidden region
$\fcal_E$ , this critical point set has dimension $0$ (see Lemma
\ref{L:Forbidden Crits}). The amplitudes in
\eqref{E:Schematic 1} and \eqref{E:Schematic 2} therefore cancel to order $h^{-2}$, and
their $h^{-1}$ term together with the
 \eqref{E:Schematic 0}'s $h^{-1}$ term contribute to $\W_{x,E}.$ In
contrast, in the allowed region $\acal_E$,
the critical point set of $S$ has dimension $d-1$ (see Lemma \ref{L:Allowed
Crits}). Each of the integrals in
\eqref{E:Schematic 2} vanishes when localized to the critical point set (see
Equation \eqref{E:Deriv Amp}). The
$h^{-2}$ contribution from \eqref{E:Schematic 1} gives the leading order of
growth for $\W_{x,E}$
in $\acal_E$.

Before giving the necessary background to prove Theorem \ref{T:Main}, let us
emphasize that our result does not
cover the case of $\abs{x}\in \set{0,\sqrt{2E}}.$ Our model has the
$SO(d-1)$ symmetry and the fixed point $x=0$
is special. All odd degree Hermite functions vanish at $x = 0$ (for odd
$|\alpha|$ the eigenfunctions are odd
polynomials times the Gaussian factor). The Kac-Rice formula becomes
singular there since $\Pi_{h_N, E}(x,x) =0 $
when $x = 0$. When $N$ is even, $d_x \Pi_N(x,x) = 0$ at $x = 0$.

\begin{figure}
        \centering
        \begin{subfigure}[b]{0.5\textwidth}
                \includegraphics[width=\textwidth]{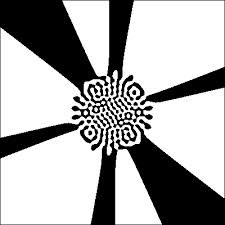}
                \label{fig:gull}
        \end{subfigure}%
        ~ 
       \begin{subfigure}[b]{0.5\textwidth}
                \includegraphics[width=\textwidth]{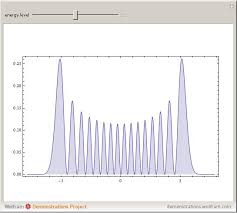}
                \label{fig:mouse}
        \end{subfigure}
\caption{The boundary between the white and black region is the nodal set of
a random Hermite function in
dimension $2$ is shown on the left. The figure on the right shows the graph
of a Hermite function in dimension
$1.$}
\end{figure}

The caustic set $\abs{x} = \sqrt{2E}$ is also special. It is the image of
the projection $\pi: \{H = E\} \to
\R^d$ along its singular set, where the projection has a fold singularity.
As discussed in \cite{KT}  (see also
\cite{T}), this fold singularity causes a blow-up in $L^p$ norms of
eigenfunctions around the caustic set (as
illustrated in the second figure in dimension one).

The caustic also causes anomalous behavior of the nodal set in a small
`boundary layer' around $\partial
\acal_E$. The nodal hypersurfaces in the forbidden region always cross the
caustic set and connect with nodal
hypersurfaces in the allowed region. In subsequent work we plan to rescale
the nodal sets in
$h^{2/3}$-neighborhoods of $\partial \acal_E$  and study their scaled
distribution.

In semi-classical $h$-notation, the Donnelly-Fefferman result is that for
Laplace eigenfunctions of real analytic
compact Riemannian manifolds, with  $h^{-1}$ the eigenvalue of
$\sqrt{\Delta}$,
 $$c_{g}
h^{-1}  \leq \hcal^{d-1}(Z_{\phi_{h}}) \leq   C_{g}
h^{-1}. $$
Thus in the allowed region, the order of magnitude of the nodal set is the
same as for Laplace eigenfunctions.
As noted above, this has been proved for eigenfunctions of Schr\"odinger
operators with real analytic metrics and
potentials in \cite{J}. The order of magnitude $h^{-\half}$ in the forbidden
region is a new result.  We hope to
explain this result
deterministically  in  subsequent work. It is evident from the graphics that
the nodal domains in $\fcal_E$ have
some angular structure and that the `frequency' of eigenfunctions in the
forbidden region is lower than in
the allowed region.

We expect that the results of this article generalize to all semi-classical
Schr\"odinger operators with
potentials of quadratic
at infinity with evident modifications. In place of eigenspaces one would
take linear combinations of
eigenfunctions
with eigenvalues from intervals of width $O(h)$ corresponding to a fixed
energy level. The case of radial
potentials
should be especially similar. But the difference ``frequencies" of nodal
sets in the allowed and forbidden
regions
should be a general phenomenon. We hope to take this up in subsequent work.
It would also be interesting to
generalize the methods and results of \cite{NS} to random Hermite
eigenfunctions.

Thanks to Long Jin for spotting a gap in the original version of this article, which led
to a substantial revision of  \S \ref{S:Forbidden Proof}.

\section{Background}\label{S:Background}

The calculation of the expected distribution of zeros is based on the
Kac-Rice formula. In this formula the
density of zeros of a Gaussian random function is expressed in terms of the
covariance function
\begin{equation} \label{COV}
\Pi_{h_N, E} (x, y) : = \EE( ( \Phi_{N} (x)
\Phi_N(y)):=\sum_{\abs{\alpha}=N} \phi_{\alpha,h_N}(x)
\phi_{\alpha,h_N}(y),   \end{equation}
which (as is well known) is the orthogonal projection onto the eigenspace
$V_N$. We will often write
$\Pi_{h,E}=\Pi_{h_N,E}.$ As in the case of spherical
harmonics, a key input into the calculations is a relatively explicit
formula for $\Pi_{h_N,E}$.  In this
section, we review the Mehler formulae and then the Kac-Rice formula.
Further background may be found in
\cite{AT,BSZ}.

\subsection{Mehler Formula}
The Mehler formula is an explicit formula for the Schwartz kernel $U_h(t,
x,y)$  of the propagator,
 $e^{-\ihbar t H_h}.$ The Mehler formula \cite{F} reads
\begin{equation}
 U_h(t, x,y) =e^{-\ihbar t H_h}(x,y)= \frac{1}{(2\pi i h \sin t)^{d/2}}
 \exp\left( \frac{i}{h}\left(
 \frac{\abs{x}^2 + \abs{y}^2}{2} \frac{\cos t}{\sin t} - \frac{x\cdot
 y}{\sin t} \right) \right),
 \label{E:Mehler}
\end{equation}
where $t \in \R$ and $x,y \in \R^d$. The right hand side is singular at
$t=0.$ It is well-defined as a
distribution, however, with $t$ understood as $t-i0$. Indeed, since $H_h$
has a positive spectrum the propagator
$U_h$ is holomorphic in
the lower half-plane and $U_h(t, x, y)$ is the boundary value of a
holomorphic function in $\{\Im t < 0\}$.

In the future, we write
\be \label{Mehler-S}
 S(t,x,y) =\frac{\abs{x}^2 + \abs{y}^2}{2} \frac{\cos t}{\sin t} -
 \frac{x\cdot y}{\sin t}
 \ee
for the phase in the Mehler formula \eqref{E:Mehler}.

\subsection{Spectral projections}
The second fact we use is that the spectrum of $H_h$ is easily related to
the integers $|\alpha|$. The operator
with the same eigenfunctions as $H_h$ and eigenvalues $h |\alpha|$ is often
called the number operator, $h
\ncal$. If we replace $U_h(t)$ by $e^{- \frac{i t}{h} \ncal}$ then the
spectral projections $\Pi_{h, E}$ are
simply the Fourier coefficients of $e^{- \frac{i t}{h} \ncal}$. In Lemma
\ref{L:Resolved Projector} we will
derive the related formula,
\begin{align}
\label{E:Projector Integral Forbidden}
\Pi_{h_N, E}(x,y)&=\int_{-\pi}^{\pi} U_h(t-i\epsilon,x,y) e^{\ihbar
(t-i\epsilon) E} \frac{dt}{2\pi}.
\end{align}
 The integral is independent of $\epsilon$. Using the Mehler formula
 \eqref{E:Mehler} we obtain a rather explicit
 integral representation of \eqref{COV}.

\subsection{Kac-Rice Formula}\label{S:KR}  Next we recall the Kac-Rice
formula. We refer to \cite{BSZ, AT} for
further background and proofs of the Kac-Rice formula in a general context
that applies to the setting of this
article. In fact we state the result on a general manifold for future
applications to more general Schr\"odinger
operators.

Let $(M,g)$ be a smooth Riemannian manifold of dimension $m$ and $dV_g$ be
the induced volume form on $M.$
Consider $f:M\gives \R,$ a smooth random function so that at each $x\in M$
the density $\Den_{f(x)}$ with respect
to Lebesgue measure exists. Let us write $\abs{Z_f}$ for the (random)
hypersurface measure on the nodal set
$f^{-1}(0).$
\begin{proposition}[Kac-Rice] \label{L:KR} $\E{\abs{Z_f}}$ has a density $F$
with respect to $dV_g$ given by
  \begin{equation}
F(x)=\Den_{f(x)}(0)\cdot \E{\abs{d f(x)}_g\,|\, f(x)=0}.\label{E:General
KR}
\end{equation}
\end{proposition}
\noindent In order to rewrite this expression for $F$ for our purposes,
suppose that $f$ is a centered
$1-$dimensional Gaussian field on $M.$ This means that for every $x\in M,$
the random variable $f(x)$ is a
real-valued Gaussian with mean $0.$ Recall that the covariance kernel of $f$
is defined by
\[\Pi_f(x,y):=\E{f(x)f(y)}.\]
The law of any centered Gaussian field on $M$ is determined uniquely by its
covariance kernel. In particular, we
may rewrite the general Kac-Rice formula of Lemma \ref{L:KR} only in terms
of $\Pi_f(x,y)$ as follows.
\begin{Lem}[Kac-Rice for Gaussian Fields]\label{L:Gaussian KR}
Let $f$ be a smooth centered Gaussian field on $M.$ Fix $x\in M.$ In a
geodesic normal coordinate chart centered
at $x,$
\begin{equation}
F(x)= \lr{2\pi}^{-\frac{d+1}{2}}\int_{\R^d}|\W_x^{1/2}\xi|
e^{-\abs{\xi}^2/2}d\xi,\label{E:Gaussian KR}
\end{equation}
where $\W_x$ is the $d\x d$ matrix
\begin{align}
\notag \lr{\W_x}_{1\leq j,k\leq d} &= \dell_{x_j}\dell_{y_k}|_{x=y} \log
\Pi_f(x,y)\\
\label{E:Gaussian Cov Mat} &=
\frac{\Pi_f(x,x)\dell_{x_k}\dell_{y_j}|_{x=y}\Pi_f(x,y)-\dell_{x_k}|_{x=y}\Pi_f(x,y)\cdot
\dell_{y_j}|_{x=y}\Pi_f(x,y)}{\Pi_f(x,x)^2}.
\end{align}
\end{Lem}
\begin{proof}
Fix $x\in M.$ The pair $\lr{f(x),df(x)}$ is a centered Gaussian vector. The
so-called regression formula states
that if $(v,w)$ is any centered Gaussian vector with covariance
\[\Cov(v,w)=\twomat{A}{B}{B^*}{C},\]
then $w$ conditioned on $v=0$ is again a centered Gaussian with covariance
$C-B^* A^{-1}B.$ For the vector
$\lr{f(x),df(x)},$ we have
\[A=\Pi_f(x,y),\quad (B)_{1, 1\leq j\leq d}=
\dell_{x_j}|_{x=y}\Pi_f(x,y),\qquad \lr{C}_{1\leq k,j\leq d}=
\dell_{x_j}\dell_{y_k}|_{x=y}\Pi_f(x,y).\]
Hence, the vector $df(x)$ conditioned on $f(x)=0$ is a centered Gaussian
vector with covariance matrix
\[\frac{\Pi_f(x,y)\dell_{x_j}\dell_{y_k}|_{x=y}\Pi_f(x,y) -
\dell_{x_j}|_{x=y}\Pi_f(x,y)
\dell_{y_k}|_{x=y}\Pi_f(x,y)}{\Pi_f(x,x)},\]
which equals $\Pi_f(x,x)\cdot \dell_{x_j}\dell_{y_k}|_{x=y} \log
\Pi_f(x,y).$ Note that
\[\Den_{f(x)}(0)=\lr{2\pi\Pi_f(x,x)}^{-1/2}.\]
Observe that $\Pi_f(x,x)^{-1/2}\cdot df(x)$ is a centered Gaussian vector
with covariance matrix
\[(\W_x)_{jk}=\dell_{x_j}\dell_{y_k}|_{x=y} \log \Pi_f(x,y).\]
Although the matrix $\W_x$ is non-negative definite, it need not be positive definite. Up to an orthogonal change of coordinates, we may write it as
\[\W_x = \twomat{\twiddle{\W}_x}{0}{0}{0}\]
for some positive definite matrix $\twiddle{\W}_x$ matrix of size $k\x k$ for some $1\leq k\leq n.$ The density of a centered Gaussian $\eta$ on $\R^k$ with positive definite covariance matrix $\twiddle{\W}_x$ is then given by
\[\frac{1}{\lr{2\pi}^{k/2}\det \twiddle{\W}_x
^{1/2}}e^{-\frac{1}{2}\inprod{\twiddle{\W}_x^{-1}\eta}{\eta}}d\eta .\]
Thus, using \eqref{E:General KR}, we find
\begin{align*}
F(x)&=\lr{2\pi}^{-1/2}\E{\Pi_f(x,x)^{-1/2}\cdot \abs{d f(x)}_g\,|\,
f(x)=0}\\
      &=\lr{2\pi}^{-1/2} \int_{\R^k}
      \frac{\abs{\eta}}{\lr{2\pi}^{k/2}\det
      \twiddle{\W}_x^{1/2}}e^{-\frac{1}{2}\inprod{\twiddle{\W}_x^{-1}\eta}{\eta}}d\eta\\
      &= \lr{2\pi}^{-\frac{k+1}{2}}\int_{\R^k} |\twiddle{\W}_x^{1/2}\tilde\xi|
      e^{-\abs{\tilde\xi}^2/2}d\tilde\xi\\
      &= \lr{2\pi}^{-\frac{d+1}{2}}\int_{\R^d} |\W_x^{1/2}\xi|
      e^{-\abs{\xi}^2/2}d\xi,
\end{align*}
as claimed.
\end{proof}

Let us denote
$\w_{d-1}=Vol(S^{d-1})=\frac{2\pi^{d/2}}{\Gamma\lr{\frac{d}{2}}}.$ In the
course of proving Theorem
\ref{T:Main}, we will need the following identity for the expected value of
the absolute value of a standard
Gaussian:
\begin{align}
\notag
\int_{\R^d}\frac{\abs{v}}{\lr{2\pi}^{d/2}}e^{-\frac{\abs{v}^2}{2}}dv&=\frac{\w_{d-1}}{\lr{2\pi}^{d/2}}\cdot
\int_0^{\infty} r^d e^{-\frac{r^2}{2}}\\
              \label{E:Gaussian Exp}                     &=
              \sqrt{2}\frac{\Gamma\lr{\frac{d+1}{2}}}{\Gamma\lr{\frac{d}{2}}}.
\end{align}

\subsection{Stationary Phase with Non-Degenerate Critical
Manifolds}\label{S:SP}
We will also need the
 method of stationary phase for non-degenerate critical manifolds, and
 recall the statement here. For further
 background we refer to \cite{DSj, GrSj, Hor, Zw}. Let $S, a \in
 C^{\infty}(\R^N),$ and consider
\[I(h)=\lr{2\pi ih}^{-N/2}\int_{\R^N}e^{iS(x)/h}a(x)dx.\]
One says that $S$ is Bott-Morse if the critical points of $S$ form a
non-degenerate critical manifold, i.e. the
transverse Hessian is non-degenerate.
\begin{Lem} \label{L:SP}
  If $S$ is a Bott-Morse function with connected critical manifold $W$ of
  dimension $n$, then there are constants
  $c_j$ such that
  \begin{equation}
I(h)=\lr{2\pi ih}^{-n/2}e^{-\frac{1}{2}\pi i
\nu}e^{iS(W)/h}\lr{\sum_{k=0}^{\infty}c_k
h^k}+O(h^{\infty})\label{E:Sp Expansion}
\end{equation}
with
\begin{equation}
c_0 = \int_W a(y)d\mu_W,\label{E:Leading Term}
\end{equation}
where $\nu$ is the Morse index of $S$ along $W$ and $d\mu_W$ is the Leray
measure on $W$ induced by its defining
function $dS.$
\end{Lem}
Equivalently, $d\mu_W$ is the quotient of $\abs{dx}$ by the Riemannian
measure on the normal bundle associated to
$S:$
\[d\mu_W=\frac{\abs{dx}}{\abs{\det Hess^{\perp} S}^{1/2}\abs{dz}},\]
where $Hess^{\perp} S$ is the normal Hessian of $S$ and $z$ is a coordinate
on the normal bundle to $S.$ In addition to Lemma \ref{L:SP}, we will need an explicit
expression for the sub-leading terms of the stationary phase expansion in the case when
$S$ is quadratic and has a single non-degenerate critical point.

\begin{Lem}[\cite{Hor} Theorem 7.7.5]\label{L:SP LO}
Suppose $a,S\in \curly S(\R)$ and $S$ is a complex-valued phase function such that $\Im
S|_{\supp(a)} \geq 0$ with a unique non-degenerate critical point at $t_0\in \supp(a)$
satisfying $\Im S(t_0)=0.$ Then
\begin{align}
\label{E:Stationary Phase 2}  I(h) = C(S) \left[a(t_0)+\frac{h}{i}\lr{-\frac{a''}{2
S''}+\frac{S''''\cdot a}{8\lr{S''}^2}+\frac{S'''\cdot a'}{2\lr{S''}^2}-\frac{5\lr{S'''}^2
\cdot a}{24 \lr{S''}^3}}\bigg|_{t=t_0}+O(h^2)\right],
\end{align}
where
\begin{equation}
C (S)= e^{i \frac{\pi}{4} \text{sgn} S''(0)} \lr{\frac{2\pi
h}{\abs{S''(t_0)}}}^{1/2}.\label{E:C}
\end{equation}
\end{Lem}

\section{Semi-Classical Propagator and spectral projections} As mentioned in the beginning
of \S
\ref{S:Background}, the covariance kernel
of the Gaussian field $\Phi_N$ is $\Pi_{h_N, E} (x, y),$ the kernel of the
spectral projector for $H_h$ onto the
$E-$eigenspace. The Kac-Rice formula \eqref{E:Gaussian KR} and equation
\eqref{E:Gaussian Cov Mat} show the
density of zeros of $\Phi_N$ is controlled by $\Pi_{h_N,E}(x,y)$ and its
derivatives evaluated on the diagonal
$x=y.$ The main result  of this section (Proposition \ref{BIGPROP})  gives a
representation
$\Pi_{h_N,E}$ as a semi-classical oscillatory integral.

First we use the periodicity of the propagator $U_h(t)$ to express the spectral
projections as Fourier coefficients of the propagator as in  \eqref{E:Projector
Integral Forbidden}.

\begin{Lem}\label{L:Resolved Projector}
For each $N,$ we abbreviate $h_N=h.$ For every $\ep> 0,$ we may write
\begin{align}
\label{E:Projector Integral Forbiddenb}
\Pi_{h_N, E}(x,y)&=\int_{-\pi}^{\pi} U_h(t-i\epsilon,x,y) e^{\ihbar
(t-i\epsilon) E} \frac{dt}{2\pi},
\end{align}
where $U_h$ is defined in \eqref{E:Mehler}. The integral is independent of
$\epsilon$.
\end{Lem}

\begin{proof}
 From the
definition of the kernel of $e^{-\ihbar t
H}$, we have
\[ U_h(t-i\epsilon,x,y) = \sum_{\alpha} e^{-\ihbar (t-i\epsilon)
h(|\alpha|+d/2)} \phi_{\alpha,h}(x)
\phi_{\alpha,h}(y), \]
where $\alpha \in \Z_{\geq0}^d$ is a multi-index, and $\epsilon>0$ ensures
the absolute convergence of the series
for fixed $x,y$. Using that $E=h_N(N+d/2)$, we get
\begin{eqnarray*}
 \int_{-\pi}^{\pi} U_h(t-i\epsilon,x,y) e^{\ihbar (t-i\epsilon) E}
 \frac{dt}{2\pi} &=&
 \int_{-\pi}^{\pi}\sum_{\alpha} e^{-i (t-i\epsilon) (|\alpha|-N)}
 \phi_{\alpha,h}(x) \phi_{\alpha,h}(y)
 \frac{dt}{2\pi} \\
 &=& \sum_{\alpha}  \phi_{\alpha,h}(x) \phi_{\alpha,h}(y)\int_{-\pi}^{\pi}
 e^{-i (t-i\epsilon)
 (|\alpha|-N)}\frac{dt}{2\pi} \\
 &=& \sum_{|\alpha|=N}  \phi_{\alpha,h}(x) \phi_{\alpha,h}(y) =
 \Pi_{h,E}(x,y).
\end{eqnarray*}

\end{proof}

We then use Mehler's formula \eqref{E:Mehler} to obtain an oscillatory integral formula.
It will prove to be convenient to break up the integral using two cutoff functions.
First, for
any $\delta\in \lr{0,
\frac{\pi}{8}},$ define a smooth function $\chi_{\delta}:S^1\gives [0,1]$
satisfying
\[\chi_{\delta}(t)=
\begin{cases}
  \case{1}{t\in \lr{-\delta,\delta}\cup \lr{\pi-\delta,\pi+\delta}}\\
  \case{0}{t\not \in
  \lr{-2\delta,2\delta}\cup \lr{\pi-2\delta,\pi+2\delta}}
\end{cases}.
\]
where $t \in S^1=\R/2\pi\Z$. Second, define the smooth function
$\twiddle{\chi}_E: \R^d\gives [0,1]$ satisfying
\[\twiddle{\chi}_E(p)=
\begin{cases}
  \case{1}{\abs{p}\leq 3E}\\ \case{0}{\abs{p}>4E}
\end{cases}.
\]

We fix $\ep\in (0,1)$ and combine Lemma \ref{L:Resolved Projector} with  \eqref{E:Mehler}
to obtain,
\begin{align} \label{E:main-intg}
\Pi_{h,E}(x,y)
 &=\lim_{\ep \to 0^+}\int_0^{2\pi}
\frac{e^{-i\lr{t-i\ep}d/2}\chi_\delta(t)}{(2\pi i h \sin
(t-i\epsilon))^{d/2}} e^{\ihbar \lr{S(t-i\epsilon,x,y)+(t-i\epsilon)
E}}\frac{dt}{2\pi} \\
\notag &+ \lim_{\ep \to
0^+}\int_0^{2\pi}\frac{\lr{1-\chi_\delta(t)}e^{-i\lr{t-i\ep}d/2}}{(2\pi i h
\sin
(t-i\epsilon))^{d/2}} e^{\ihbar \lr{S(t-i\epsilon,x,y)+(t-i\epsilon)
E}}\frac{dt}{2\pi}
\end{align}
The second equal sign is valid, since the integral in the first line is
independent of $\epsilon$.
The first term is problematic due to the singularity of the integrand at $t = 0$. We
therefore rewrite it by taking
the  Fourier transform of the
Mehler formula in the $y$ variable.
We also define the  integration-by-parts operator
\begin{equation}
L_{h,t}:=\frac{h}{i}\frac{1}{E+\dell_t
\widehat{S}(t,x,p)}\dell_t\label{E:Lt}
\end{equation}
and will write $L_{h,t}^*$ for its adjoint.  We now give the representation of $
\Pi_{h_N, E}(x,y)$
as a semi-classical oscillatory integral operator that will be used in the Kac-Rice
calculations.

\begin{proposition} \label{BIGPROP} Further, for every $\delta
\in \lr{0,\frac{\pi}{8}}$ and each integer $k>d+1$
\begin{align}
\label{E:Projector Integral Allowed}  \Pi_{h_N, E}(x,y)&= \int_{S^1\x \R^d}
\frac{\chi_{\delta}(t)}{\lr{\cos
t}^{d/2}} e^{\frac{i}{h}\lr{\widehat{S}(t,x,p)-y\cdot p + t
E}}\,\,\twiddle{\chi}_R(p)\frac{dp}{\lr{2\pi
h}^d}\frac{dt}{2\pi}\\
\notag &+ \int_{S^1\x
\R^d}\left[\lr{L_{h,t}^*}^k\lr{\frac{\chi_{\delta}(t)}{\lr{\cos
t}^{d/2}}}\right]
e^{\frac{i}{h}\lr{\widehat{S}(t,x,p)-y\cdot p + tE}}\,
\lr{1-\twiddle{\chi}_R(p)}\frac{dp}{\lr{2\pi
h}^d}\frac{dt}{2\pi}\\
\notag &+ \int_{S^1} \frac{1-\chi_{\delta}(t)}{\lr{2\pi h i \sin
t}^{d/2}}\cdot e^{\frac{i}{h}\lr{S(t,x,y)+tE}}
\frac{dt}{2\pi}.
\end{align}
\end{proposition}
\begin{remark} We will prove in Lemma
\ref{L:IBP Decay} below that
\begin{equation}
\abs{E+\dell_t \widehat{S}(t,x,p)}\geq c \lr{1+\abs{p}^2}\label{E:IBP
Decay}
\end{equation}
with $c>0$ when $\abs{x}^2+\abs{p}^2\geq 3E.$ In particular,
$\lr{1-\twiddle{\chi}_E(p)}L_{h,t}$ is a
well-defined operator.
  Lemma \ref{L:IBP Decay} shows that the second term in \eqref{E:Projector
  Integral Allowed} is well-defined.
\end{remark}

\begin{proof}

 Let $\fcal_{h}$ denote the semiclassical
Fourier transform.  $U_h(t-i\epsilon,x,y)$ is a Schwartz
function in the $y$ variable and we may take its Fourier transform
\begin{align}
\notag \widehat{U}_h(t-i\epsilon,x,p):=& \fcal_{h,y \to p}
U_h(t-i\epsilon,x,y).
\end{align}

\begin{Lem}
Fix any $\epsilon \in (0,1).$  The
semiclassical Fourier transform of $U_h$ in $y$ is
\begin{align}
\notag
 \widehat{U}_h(t-i\epsilon,x,p)
\label{E:Propogate FT}=& \frac{1}{(\cos (t-i\epsilon))^{d/2}} e^{\ihbar
\widehat{S}(t-i\ep,x,p)},
\end{align}
where
\begin{equation}
\widehat{S}(t,x,p)=-\frac{\abs{x}^2+\abs{p}^2}{2}\frac{\sin t}{\cos t} +
\frac{x\cdot p}{\cos t}.\label{E:Dual
Phase}
\end{equation}
\end{Lem}

\begin{proof}
We will use the Mehler Formula \eqref{E:Mehler} for $U_h$. First note that
the prefactor $\lr{2\pi i h
\sin\lr{t-i\ep}}$ never vanishes and that $U_h(t-i\ep, x, y)$ is a smooth
function. To show $U_h(t-i\ep,x,y)$ has
fast decay in $y$ under the condition in the Lemma, it suffices to check
that $\Im S(t-i\ep,x,y)>c y^2$ for some
$c(\epsilon)>0$ and all large enough $y$. From the definition of $S$, we
have
\[\Im S(t-i\ep,x,y) = \frac{\abs{x}^2+\abs{y}^2}{2} \Im (\cot(t-i\ep))-
x\cdot y\, \Im(\csc(t-i\ep)).\]
Note that
\[\Im\lr{\cot\lr{t-i\ep}}=\frac{e^{2\ep}-e^{-2\ep}}{\abs{e^{\ep}e^{it}-e^{-\ep}e^{-it}}^2}>0.\]
We thus have that $U(t-i\ep, x,y)$ is a Schwartz function for all $t$ with
the Schwartz seminorms depending only
on $\ep.$ Thus, the Fourier transform in $y$ is well-defined.

To get explicit formula for $\widehat{U}_h(t-i\ep,x,p)$, recall that if $Q$
is a $d\x d$ non-degenerate symmetric
real matrix, then
\[\curly F_{y\gives p}\lr{e^{\frac{i}{2h}\inprod{Qy}{y}}}(p)=\frac{\lr{2\pi
h }^{d/2}e^{i \pi \sgn(Q)/4}}{|\det
Q|^{1/2}}e^{-\frac{i}{2h}\inprod{Q^{-1}p}{p}}.\]
Thus, using (\ref{E:Mehler}),
\begin{align*}
  \curly F_{y\gives p} U_h(t-i\ep, x,p) & = \frac{\lr{2\pi i
  h}^{d/2}}{\lr{\cot (t-i\ep)}^{d/2}}\cdot \lr{2\pi i
  h \sin (t-i\ep)}^{-d/2}\cdot e^{-i\frac{\tan (t-i\ep)}{2h} \abs{x}^2}
  \cdot e^{\frac{i}{h}\frac{x\cdot p}{\cos
  (t-i\ep)}}\cdot e^{-\frac{i}{2h}\tan (t-i\ep)\abs{p}^2}\\
&=\lr{\cos
(t-i\ep)}^{-d/2}e^{-\frac{i}{h}\lr{\frac{\abs{x}^2+\abs{p}^2}{2}\tan
(t-i\ep) -\frac{x\cdot p}{\cos
(t-i\ep)}}},
\end{align*}
proving  the Lemma.
\end{proof}
Using the Fourier inversion formula, we may write the first term in
\eqref{E:main-intg} as
\[
\lim_{\ep \to 0^+}\int_0^{2\pi} \int_{\R^d} \frac{\chi_\delta(t)}{( \cos
(t-i\epsilon))^{d/2}} e^{\ihbar
\lr{\widehat{S}(t-i\epsilon,x,y)-y\cdot p+(t-i\epsilon) E}}\frac{dp}{(2\pi
h)^d}\frac{dt}{2\pi}
\]
Writing $1=\twiddle{\chi}_E(p)+\lr{1-\twiddle{\chi}_E(p)},$ the integral in
the previous line becomes
\begin{align}
\label{E:Int I}&\lim_{\ep \to 0^+}\int_0^{2\pi} \int_{\R^d}
\frac{\chi_\delta(t) }{( \cos (t-i\epsilon))^{d/2}}
e^{\ihbar \lr{\widehat{S}(t-i\epsilon,x,p)-y\cdot p+(t-i\epsilon)
E}}\twiddle{\chi}_E(p)\frac{dp}{(2\pi
h)^d}\frac{dt}{2\pi}\\
\label{E:Int II}+& \lim_{\ep \to 0^+}\int_0^{2\pi} \int_{\R^d}
\frac{\chi_\delta(t) }{( \cos
(t-i\epsilon))^{d/2}} e^{\ihbar \lr{\widehat{S}(t-i\epsilon,x,p)-y\cdot
p+(t-i\epsilon)
E}}\lr{1-\twiddle{\chi}_E(p)}\frac{dp}{(2\pi h)^d}\frac{dt}{2\pi}.
\end{align}
The dominated convergence theorem allows us to take the limit as $\ep\gives
0$ in \eqref{E:Int I}, which gives
the first term in \eqref{E:Projector Integral Allowed}. Finally, to study
\eqref{E:Int II}, we need the following
result.
\begin{Lem}\label{L:IBP Decay}
  For each integer $k\geq 0,$
\[(L_{h,t}^*)^k = h^k\sum_{i=0}^{k} a_{k,i}(t,x,p)  \dell_t^i \]
with $\abs{a_{k,i}(t,x,p)}\leq c_{k,i} \lr{1+\abs{p}^2}^{-k}$ for all
$\abs{p}>3E$ and some $c>0.$ The operator
$L_{h,t}^*$ is the adjoint of $L_{h,t},$ which is defined in \eqref{E:Lt}.
\end{Lem}
\begin{proof}
Let us write
\[f(t,x,p,E):=E+\partial_t \widehat{S}=  E -
\frac{\abs{x}^2+\abs{p}^2}{2\cos^2 t} + \frac{x \cdot p
\sin(t)}{\cos^2 t}.\]
Note that when $t\in \supp \chi_{\delta},$ we have that
\begin{equation}
\abs{f(t,x,p,E)}>c\lr{1+\abs{p}^2}\label{E:f Lower Bound}
\end{equation}
for some $c>0$ as long as $\abs{p}>3E.$ To control $L^*$, let us write
$M_{1/f}$ for the multiplication operator
by $1/f,$ which is well-defined when $t\in \supp \chi_{\delta}$ and
$\abs{p}>3E.$ We have
\begin{eqnarray*}
(L_{h,t}^*)^k &=& \lr{\frac{h}{i}}^k \lr{\partial_t \circ M_{1/f}}^k,
\end{eqnarray*}
which we may write as a sum of finitely many terms of the form
\begin{equation}
\lr{\frac{h}{i}}^k C_{k_1,\ldots, k_r} \frac{\partial_t^{k_1}(f) \cdots
\partial_t^{k_r}(f)}{f^{k+r}}
\partial_t^{k_0}\label{E:Chaz Exp}
\end{equation}
where $k_0+\cdots+k_r=k$ and $C_{k_0,\ldots, k_r}$ are some constants. Note
that for any $j\geq 1$
\begin{equation}
\abs{\dell_t^j f(t,x,p,E)} \leq  c_j\cdot \lr{1+\abs{p}^2}\label{E:f Upper
Bound}
\end{equation}
for some constants $c_j\in \R$ that are uniform in $t\in\supp
\chi_{\delta}.$ Combining \eqref{E:f Lower
Bound}-\eqref{E:f Upper Bound} completes the proof.
\end{proof}
 Lemma \ref{L:IBP Decay} allows us
to integrate by parts using the operator $L_{h,t}^k$ in the integral
\eqref{E:Int II}. The resulting integrand is
$L^1(\R)$ uniformly in $\ep$.  We therefore send $\ep\gives 0$ and again use
the dominated convergence theorem to
obtain the second term in the stated formula \eqref{E:Projector Integral
Allowed}. This completes the proof of Proposition \ref{BIGPROP}.
\end{proof}

\begin{remark} We note that the estimate for $L_{h,t}^*$ is first used by Chazarain
\cite{Ch}, and a similar result can be
obtained for a more general class of potential with quadratic growth at
infinity.

\end{remark}

\section{Proof of Theorem \ref{T:Main}}
 Lebesgue measure on $\R^d$ is the volume form for the flat metric,  so by
the Kac-Rice formula (\ref{E:Gaussian
KR}),    $\E{\abs{Z_{\Phi_N}}}$ has a density $F_N(x)$ given by
\eqref{E:Gaussian KR}. In order to use
this formula, we need to understand the $d\x d$ matrix
\[\W_{x,E}:=\frac{\Pi_{h,E}(x,x)\dell_{x_k}\dell_{y_j}|_{x=y}\Pi_{h,E}(x,y)-\dell_{x_k}|_{x=y}\Pi_{h,E}(x,y)\cdot
\dell_{y_j}|_{x=y}\Pi_{h,E}(x,y)}{\Pi_{h,E}(x,x)^2}.\]
The main step of the proof of Theorem \ref{T:Main} is  the
following Proposition, which gives an
explicit formula for $\W_{x,E}.$

\begin{proposition}\label{P:Cov Matrix Asymptotics}
Suppose $\abs{x}^2\in \curly F_E.$ Then
  \begin{equation}
    \label{E:Cov Matrix Forbidden}
    \lr{\W_{x,E}}_{kj}=h^{-1}\frac{(\delta_{kj}-\hat{x}_k\hat{x}_j)
    E}{\abs{x}\sqrt{\abs{x}^2-2E}} +O(1),
  \end{equation}
where $\hat{x}_k:=\frac{x_k}{\abs{x}.}$ Suppose $x\in \curly A_E\backslash \set{0}.$ Then
  \begin{equation}
    \label{E:Cov Matrix Allowed}
    \lr{\W_{x,E}}_{kj}= h^{-2}\delta_{kj}\cdot \frac{\w_{d-2}}{d\cdot
    \w_{d-1}}\lr{2E-\abs{x}^2}\lr{1+O(h^2)}
  \end{equation} The implied constants in the `$O$' error terms in \eqref{E:Cov Matrix
  Allowed} and \eqref{E:Cov Matrix Forbidden}
are uniform on compact subsets of the interiors in $\curly A_E\backslash\set{0}$ and
$\curly F_E.$
\end{proposition}
\noindent Theorem \ref{T:Main} follows easily by substituting \eqref{E:Cov
Matrix Allowed} and \eqref{E:Cov
Matrix Forbidden} into \eqref{E:Gaussian KR} and using the identities
\eqref{E:Gaussian Exp}. We will prove \eqref{E:Cov Matrix Forbidden} in \S
\ref{S:Forbidden Proof} and \eqref{E:Cov Matrix Allowed} in \S\ref{S:Allowed Proof}.

\subsection{ Proof of Proposition \ref{P:Cov Matrix Asymptotics} in the
Forbidden Region}\label{S:Forbidden Proof}
We fix $x\in \R^d$ with $\abs{x}^2>2E.$ Our goal is to prove Equation
\eqref{E:Cov Matrix Forbidden}. Recall from (\ref{E:Projector Integral Forbidden}) and
\eqref{E:Mehler} that
\[\Pi_{h_N, E}(x,y)=\int_{-\pi-i\epsilon}^{\pi-i\epsilon}\frac{1}{(2\pi i h \sin t)^{d/2}}
e^{\ihbar
\lr{S(t,x,y)+t E}}\frac{dt}{2\pi}\]
is an absolutely convergent integral for all $\epsilon>0$ that is independent of the value
of $\ep.$ Equation \eqref{E:Cov Matrix Forbidden} will follow from Lemma \ref{L:Forbidden
Crits}, Equations \eqref{E:Phase Derivs 1}-\eqref{E:Phase Derivs 2}, and Lemma
\ref{L:props}.

\begin{Lem}\label{L:Forbidden Crits}
The phase $S(t,x,x)+Et$ has no critical points in the real domain. In the
complex domain, it has two critical
points $\pm i\beta,$ which are the two distinct solutions to
\[\cosh\lr{\beta/2}=\frac{\abs{x}}{\sqrt{2E}}.\]
These critical points are non-degenerate.
\end{Lem}
\begin{proof}
Note that $S(t,x,x)=-\tan\lr{\frac{t}{2}}\abs{x}^2.$ Hence, $\dell_t
\lr{S(t,x,x)+Et}=0$ is equivalent to
\[\cos\lr{\frac{t}{2}}=\pm \frac{\abs{x}}{\sqrt{2E}}.\]
Since we have assumed $\frac{\abs{x}^2}{2E}>1,$ the phase $S(t,x,x)+Et$ has no
real critical points. Setting
$t=\alpha+i\beta,$ the critical point equation is equivalent to
\[\cos\lr{\frac{\alpha}{2}}\cos\lr{\frac{i\beta}{2}}-\sin\lr{\frac{\alpha}{2}}\sin\lr{\frac{i\beta}{2}}=\pm
\frac{\abs{x}}{\sqrt{2E}}.\]
  Since the right hand side is real and $\sin(it)=i\sinh(t),$ we conclude
  that $\alpha=0.$ Using that $\cosh$ is
  positive, the equation therefore reduces to
  \begin{equation}
\cosh\lr{\frac{\beta}{2}}=\frac{\abs{x}}{\sqrt{2E}},\label{E:Crit Pt Eq}
\end{equation}
which has two distinct solutions $\pm \beta$ for some $\beta>0.$ It remains
to check that these critical points
are non-degenerate:
\[\dell_{tt}S(\pm i\beta,x,x)=\mp iE\ \tanh{\beta/2},\]
which is non-zero since $\beta\neq 0.$
\end{proof}
\noindent Lemma \ref{L:Forbidden Crits} shows that to evaluate the integral
\eqref{E:Projector Integral Forbidden}, we should set $\ep=\beta.$ Let us abbreviate
\[S_{j,k}(t)=\dell_{x_j}\dell_{y_k}S(t,x,y)\big|_{x=y}\]
with the convention that $S_{j,0}(t)=\dell_{x_j}S(t,x,y)\big|_{x=y}$ and
$S_{0,0}(t)=S(t,x,x).$ We will continue to write $'$ for derivatives with respect to $t.$
\begin{Lem} \label{L:props}
For each $x\in \R^d$ and $1\leq j,k\leq d,$ we have
\begin{align}
\label{E:SP 1D Two Deriv Log} \lr{\W_{x,E}}_{j,k} &= \frac{i}{h}\left(S_{j,k}(-i\beta)-
\frac{S_{j,0}'(-i\beta) S_{0,k}'(-i\beta)}{S_{0,0}''(-i\beta)}\right) + O(1)
\end{align}
\end{Lem}
\begin{proof}
Let us write
\[\Pi_{h,E}^{j,k}=\dell_{x_j}\dell_{y_k}|_{x=y} \Pi_{h,E}(x,y),\]
again with the understanding that $\Pi_{h,E}^{j,0}$ means no derivative in $y$ and so on.
We then have
\begin{equation}
\lr{\W_{x,E}}_{j,k}=\frac{\Pi_{h,E}(x,x)\Pi_{h,E}^{j,k}-
\Pi_{h,E}^{j,0}\Pi_{h,E}^{0,k}}{\Pi_{h,E}(x,x)^2}.\label{E:Cov Mat New}
\end{equation}
Each  term in the numerator and denominator is an oscillatory integral with the same phase
$S(t,x,x)+tE.$ Indeed, if we abbreviate
\[A(t)=\frac{1}{(2\pi i h \sin t)^{d/2}}\cdot \frac{dt}{2\pi},\]
then
\begin{align}
  \label{eq:2}
\Pi_{h,E}(x,x)& =\int_{-\pi-i\epsilon}^{\pi-i\epsilon} e^{\ihbar
\lr{S(t,x,x)+t E}}A(t)\\
\Pi_{h, E}^{j,k}(x,x)&=\int_{-\pi-i\epsilon}^{\pi-i\epsilon} \lr{\frac{i}{h}S_{j,k}-
\frac{1}{h^2}S_{j,0}\cdot S_{0,k}} e^{\ihbar
\lr{S(t,x,x)+t E}}A(t)\\
\Pi_{h,E}^{j,0}(x,x)&=\int_{-\pi-i\epsilon}^{\pi-i\epsilon} \lr{\frac{i}{h}S_{j,0}}
e^{\ihbar
\lr{S(t,x,x)+t E}}A(t)\\
\Pi_{h,E}^{0,k}(x,x)&=\int_{-\pi-i\epsilon}^{\pi-i\epsilon} \lr{\frac{i}{h}S_{0,k}}
e^{\ihbar
\lr{S(t,x,x)+t E}}A(t)
\end{align}
We now apply Lemma \ref{L:SP LO} to each term. Let us rewrite \eqref{E:Stationary Phase 2}
schematically as
\[\int a e^{\frac{i}{h}S} = C\lr{a+h\cdot f_1(a,S) +O(h^2)}.\]
Here the constant $C$ depends on $S$ and $h$ and so on, but will cancel in the numerator
and denominator of \eqref{E:Cov Mat New} and the function $f_1$ is linear in the amplitude
$a$. The first term, $\Pi_{h,E}(x,x)\Pi_{h,E}^{j,k},$ in the numerator of $\W_{x,E}$ is
therefore
\begin{align*}
C^2\lr{A+h\cdot f_1(A,S)}\lr{A\left[
\frac{i}{h}S_{j,k}-\frac{1}{h^2}S_{j,0}S_{0,k}\right]+h\cdot f_1\lr{A\left[
\frac{i}{h}S_{j,k}-\frac{1}{h^2}S_{j,0}S_{0,k}\right],S}} +O(h^2),
\end{align*}
which becomes
\[C^2A^2\lr{-\frac{1}{h^2}S_{j,0}S_{0,k}+\frac{1}{h}\lr{iS_{j,k}-f_1(A,S)S_{j,0}S_{0,k}-\frac{1}{h}f_1(AS_{j,0}S_{0,k},S)}+O(1)}.\]
Similarly, the second term, $\Pi_{h,E}^{j,0}\Pi_{h,E}^{0,k},$ in the numerator of
$\W_{x,E}$ is
\[-C^2a^2\lr{\frac{1}{h^2}S_{j,0}S_{0,k}+\frac{1}{h}\left[S_{j,0}f_1(AS_{0,k},S)+S_{0,k}f_1(AS_{j,0},S)\right]+O(1)}.\]
Note that the $h^{-2}$ terms in the expansions of $\Pi_{h,E}(x,x)\Pi_{h,E}^{j,k}$ and
$\Pi_{h,E}^{j,0}\Pi_{h,E}^{0,k}$ cancel. From expression \eqref{E:Stationary Phase 2}, we
see that the terms in $f_1$ that depend on at most $1$ derivative of the amplitude will
cancel between $\Pi_{h,E}(x,x)\Pi_{h,E}^{j,k},$ and $\Pi_{h,E}^{j,0}\Pi_{h,E}^{0,k}.$
Hence, comparing the contributions of the single term in $f_1$ that involves two
derivatives of the amplitude, the numerator of \eqref{E:Cov Mat New} becomes
\[\frac{C^2A^2}{h}\lr{iS_{j,k}(-i\beta)-\frac{S_{j,0}'(-i\beta)
S_{0,k}'(-i\beta)}{S_{0,0}''(-i\beta)}+O(1)}.\]
Finally, we use  that the denominator in \eqref{E:Cov Mat New} is of the form
$C^2A^2\lr{1+O(h)}$ to complete the proof.
\end{proof}

\noindent We may use \eqref{Mehler-S} to obtain
\begin{align}
\label{E:Phase Derivs 1} S_i = -x_i \tan(\frac{t}{2}),&\quad S_i' =
-\frac{x_i}{2}\frac{1}{\cos^2(\frac{t}{2})}\\
\label{E:Phase Derivs 2} S_{i,j} =-\frac{\delta_{ij}}{\sin t},&\quad S''= -\frac{|x|^2}{2}
\frac{\sin(t/2)}{\cos^3(t/2)}.
\end{align}
\noindent Lemma \ref{L:props} now gives
\[ \lr{\W_{x,E}}_{j,k}= \frac{1}{h} \frac{\delta_{jk} -
\hat{x}_j\hat{x}_k}{\sinh(\beta)}+O(1).\]
Combining \eqref{E:Crit Pt Eq} with
\[\sinh\lr{\beta}=2\cosh\lr{\frac{\beta}{2}}\sinh\lr{\frac{\beta}{2}}=2\cosh\lr{\frac{\beta}{2}}\sqrt{\cosh^2\lr{\frac{\beta}{2}}-1}\]
proves \eqref{E:Cov Matrix Forbidden}. Before going on to prove Theorem \ref{T:Main} in
the allowed region, let us prove the following result, which we believe is of independent
interest.
\begin{Lem}[Explicit Expression of $\Pi_{h,E}$ in the Forbidden
Region]\label{L:Forbidden Derivs} Fix $\abs{x}^2>2E.$ Then, with $\beta$ defined as in
Lemma \ref{L:Forbidden Crits} and $1\leq j,k\leq d$, we have
  \begin{align}
  \label{E:Projector on Diag For}  \Pi_{h,E}(x,x) &=
  \lr{2\pi}^{-\frac{d+1}{2}}
  h^{-\frac{d-1}{2}}\frac{\abs{x}^{1/2}e^{\frac{1}{h}\lr{-\abs{x}\sqrt{\abs{x}^2-2E}+E\beta}}}{E^{1/2}\cdot
  \lr{\abs{x}^2-2E}^{1/4}\lr{\sinh \beta}^{d/2}}\lr{1+O(h)}
 \end{align}
where, as before $\hat{x}_k:=\frac{x_k}{\abs{x}}.$
\end{Lem}
\begin{proof}
Let us take $\ep=\beta$ in \eqref{E:Projector Integral Forbidden}. The real part of
the phase along the contour $[-\pi-i\beta, \pi-i\beta]$ is
\[\Re\lr{\frac{i}{h}\lr{S(t-i\beta,x,x)+(t-i\beta)E}}=\frac{\abs{x}^2}{h}\Im
\lr{\tan\lr{\frac{t-i\beta}{2}}}+\frac{\beta}{h}E,\]
which has a unique maximum when $t=0.$ We may therefore apply Lemma \ref{L:SP LO}. Let us
denote
$a(t):=\lr{i \sin t}^{-d/2}.$ We have
\begin{align*}
 \dell_{tt}|_{t=-i\beta} S(t, x,x) & = \frac{iE}{\abs{x}}\sqrt{\abs{x}^2-2E}\\
  S(t, x,x)+tE \mid_{t=-i\beta} & = i\lr{\abs{x}\sqrt{\abs{x}^2-2E}-\beta E}\\
 a(-i\beta) & = \lr{\sinh\beta}^{-d/2}=\lr{\frac{\abs{x}}{E}\sqrt{\abs{x}^2-2E}}^{-d/2}.
\end{align*}
Thus,
\begin{equation}
  \label{E:Projector}
  \Pi_{h,E}(x,x) = \lr{2\pi}^{-\frac{d+1}{2}} h^{-\frac{d-1}{2}}\frac{\lr{\sinh
  \beta}^{-d/2}\abs{x}^{1/2}}{E^{1/2}\cdot
  \lr{\abs{x}^2-2E}^{1/4}}e^{\frac{1}{h}\lr{-\abs{x}\sqrt{\abs{x}^2-2E}+E\beta}}\lr{1+O(h)},
\end{equation}
which is precisely \eqref{E:Projector on Diag For}. This completes the proof of Lemma
\ref{L:Forbidden Derivs}.
\end{proof}

\subsection{Proof of Proposition \ref{P:Cov Matrix Asymptotics} in the
Allowed Region}\label{S:Allowed Proof} The
goal of this section is to prove Equation (\ref{E:Cov Matrix Allowed}),
which is a consequence of the following
Lemma.
\begin{Lem}[Derivatives of $\Pi_{h,E}$ in the Allowed
Region]\label{L:Allowed Derivs} Suppose $0<\abs{x}^2<2E.$
Then
  \begin{align}
  \label{E:Projector on Diag} \Pi_{h,E}(x,x)&=\lr{2\pi
  h}^{-(d-1)}\lr{2E-\abs{x}^2}^{\frac{d}{2}-1}\w_{d-1}\lr{1+O(h)}\\
  \label{E:Projector Deriv on Diag}  \dell_{x_i}|_{x=y}\Pi_{h,
  E}(x,y)&=\dell_y|_{x=y}\Pi_{N,h_N}(t,x,y) =
   O(1/h^{d-1})+O(1/h^{(d+1)/2})\\
  \label{E:Projector Two Deriv on Diag}
  \dell_{x_k}\dell_{y_j}|_{x=y}\Pi_{h, E}(x,y)&= \delta_{kj}\, h^{-2}\cdot
  \lr{2E-\abs{x}^2} \cdot\frac{1}{d}\Pi_{h,E}(x,x)(1+O(h))
  \end{align}
\end{Lem}
\begin{proof}
Fix $x$ with $0<\abs{x}^2<2E.$ Equations \eqref{E:Projector on
Diag}-\eqref{E:Projector Two Deriv on Diag} are
obtained by applying stationary phase to the oscillatory integral
representation \eqref{E:Projector Integral
Allowed} for $\Pi_{h,E}$ and its derivatives. To start, note that the second
term in \eqref{E:Projector Integral
Allowed} is $O(h^{\infty})$ since we may take $k$ arbitrarily large. Also,
by Lemma \ref{L:Forbidden Crits}, we
may apply stationary phase to the third term of \eqref{E:Projector Integral
Allowed} and find that is
contribution is on the order of $h^{-\frac{d-1}{2}}.$

As we prove below, the first integral gives the leading contribution to $\Pi_{h,E}(x,x)$,
which
is on the order of $h^{-(d-1)}.$ To
see this, we will apply stationary phase. Let us compute the critical set
for the phase function
$\widehat{S}(t,x,p)-y\cdot p + Et$. To emphasize the relation between the
phase factor $\widehat{S}(t,x,p)$ and
the classical path of energy $E$ ending in time $t$ at $x$ with initial
momentum $p,$ let us write $x=x_f$ and
$p=p_i,$ where the subscripts $f$ and $i$ stand for initial and terminal
positions and momenta. We have the
relations:
\begin{equation}
x_i=\frac{x_f-p_i\sin t}{\cos t},\qquad p_f=\frac{p_i-x_f\sin t}{\cos
t}.\label{E:Classical Mechanics}
\end{equation}
The $t-$critical point equation is
\[E+\dell_t \widehat{S}(t,x_f,p_i) = E-
\frac{\lr{x_i(t,x_f,p_i)}^2+p_i^2}{2}=0\]
since $\widehat{S}$ satisfies the Hamilton-Jacobi equation associated to
$H_h.$ The $p_i$ critical point equation
is:
\[y=\dell_{p_i}\widehat{S}(t,x_f, p_i)=x_i.\]
On the diagonal, we have $y=x_f$ so that this relation is $x_i=x_f.$
Therefore, the critical manifold for
$\widehat{S}(t,x,x)-y\cdot p + Et$ is
\[W_{x,E}=\setst{\lr{t,p}\in [0,2\pi)\x T^*_x \R^d}{\pi\Phi^t(x,p)=x\quad
\text{and}\quad
\frac{\abs{p}^2+\abs{x}^2}{2}=E},\]
where $\pi:T^*\R^n \gives \R^n$ is the projection to the base and $\Phi^t$
is the Hamilton flow for the classical
harmonic oscillator $\frac{1}{2}\lr{\abs{x}^2+\abs{p}^2}.$ To apply the
method of stationary phase, we must be
sure that $\W_{x,E}$ is non-degenerate.
\begin{Lem}\label{L:Allowed Crits}
For $\delta>0$ sufficiently small, $W_{x,E}$ restricted to the support of
$\chi_{\delta}$ is
\begin{equation}
W_{x,E}\intersection supp\lr{\chi_{\delta}(t)}=\set{t=0}\x
S_{\sqrt{2E-\abs{x}^2}}^*=\set{0}\x\setst{p\in T^*_x
\R^d}{\frac{\abs{x}^2+\abs{p}^2}{2}=E},\label{E:Crit Man}
\end{equation}
which is a non-degenerate critical manifold. Moreover, the Morse index of
$\widehat{S}(t,x,p)-y\cdot p + Et$
along $W_{x,E}$ is $1.$
\end{Lem}
\begin{proof}
From the relations \eqref{E:Classical Mechanics}, we see that if $x\neq 0,$
then for $\delta$ sufficiently small,
the only value of $t$ that is in the support of $\chi_{\delta}$ for which we
may simultaneously solve
\[p_f(t,x_f,p_i)=p_i \quad \text{and}\quad x_f=x_i(t,x_f,p_i)\quad
\text{and}\quad
\frac{\abs{x}^2+\abs{p}^2}{2}=E\]
is $t=0.$ This proves \eqref{E:Crit Man}. To check that the critical
manifold is non-degenerate, let us compute
the normal Hessian of $\widehat{S}(t,x,p)-y\cdot p + Et$ along $W_{x,E}.$
Note that the fiber of the normal
bundle to $W_{x,E}\intersection supp \lr{\chi_{\delta}(t)}$ is spanned by
$\dell_t$ and $\dell_r,$ where
$r=\abs{p}$ so that $\dell_r$ is the radial vector field. We have
\begin{align*}
  \dell_{tt}\lr{\widehat{S}(t,x_f,p_i)-y\cdot p_i + t E}&= x_f\cdot p_i\\
  \dell_{tr}\lr{\widehat{S}(t,x,p)-y\cdot p + t E}&=-\abs{p_i}\\
  \dell_{rr}\lr{\widehat{S}(t,x,p)-y\cdot p + tE}&=0.
\end{align*}
Hence, the normal Hessian is
\[\twomat{x_f\cdot p_i}{-\abs{p_i}}{-\abs{p_i}}{0}.\]
The determinant is $-\abs{p_i}^2=\abs{x_f}^2-2E,$ which is non-zero as long
as $\abs{x_f}< 2E.$ Hence, $W_{x,E}$
is non-degenerate. Moreover, one easily verifies that the normal Hessian
always has one positive and one negative
eigenvalue. The Morse index of $\widehat{S}(t,x_f,p_i)-y\cdot p_i + Et$
along $W_{x,E}$ is therefore equal to
$1.$ This completes the proof of Lemma \ref{L:Allowed Crits}.
\end{proof}
Returning to the proof of Equations \eqref{E:Projector on
Diag}-\eqref{E:Projector Two Deriv on Diag}, we apply
the stationary phase method (\S \ref{S:SP}) to the first term in
\eqref{E:Projector Integral Allowed}. Writing
$r$ for the radial coordinate on $\R^d,$ we have
\[d\mu_{W_{x,E}}= \left.\frac{dt\wedge dx }{\abs{\det Hess^{\perp}
\widehat{S}}^{1/2}\cdot dt\wedge
dr}\right|_{W_{x,E}}.\]
Write $d\w$ for the uniform measure on the sphere of radius
$\sqrt{2E-\abs{x}^2}$ normalized to have volume $1$
and $\w_{d-1}$ of the volume of unit sphere in $\R^d.$ We may thus express
$dt\wedge dx$ as $\w_{d-1} dt \wedge
r^{d-1} dr \wedge d\w.$ We find that
\[d\mu_{W_{x,E}}=\lr{2E-\abs{x}^2}^{\frac{d}{2}-1}\w_{d-1}\cdot d\w.\]
Observe that the amplitude $\frac{\chi_{\delta}(t)}{\cos
t^{-d/2}}\twiddle{\chi}_R(p)$ in the first integral of
(\ref{E:Projector Integral Allowed}) is identically equal to $1$ on
$W_{x,E}.$ Hence its integral over with
respect to $d\mu_{W_{x,E}}$ is
\[\lr{2E-\abs{x}^2}^{\frac{d}{2}-1}\w_{d-1}.\]
Noting that $\lr{\hat{S}(t,x,p)-y\cdot p +t E}|_{M_0}=0$ we obtain from
Lemma \ref{L:SP} that $\Pi_{h,E}(x,x)$ may
written as
\begin{equation}
\lr{2\pi
h}^{-(d-1)}\lr{2E-\abs{x}^2}^{\frac{d}{2}-1}\w_{d-1}\lr{1+O(h)}.\label{E:Allowed
Amp}
\end{equation}
This confirms Equation (\ref{E:Projector on Diag}).

In order to compute the
asymptotics of
$\dell_{x_k}|_{x=y}\Pi_{h,E}(x,y),$ we argue in a similar fashion. First, we
differentiate under the integral in
all three terms of \eqref{E:Projector Integral Allowed}. Just as before, the
second term has order
$O(h^{\infty}),$ and the third term is $O(1/h^{\frac{d+1}{2}}).$
The first term, whose leading order term $O(1/h^{d})$ vanishes, will at most give a
$O(1/h^{d-1})$ contribution. To prove this, note that when $t=0,$
\begin{equation}
\dell_{x_k}|_{x=y}\lr{\widehat{S}(t,x,p)-y\cdot p + Et}= p_k,\label{E:Deriv
Amp}
\end{equation}
whose integral over $W_{x,E}$ vanishes. Thus, applying Lemma \ref{L:SP}, we
find that
\begin{align*}
  \dell_{x_k}|_{x=y}\Pi_{h,E}(x,y)&= O(\Pi_{h,E}(x,x)).
\end{align*}
This proves Equation (\ref{E:Projector Deriv on Diag}). It remains to study
$\dell_{x_k}\dell_{y_j}|_{x=y}\Pi_{h,E}(x,y).$ Like before, we differentiate
under the integral sign in
\eqref{E:Projector Integral Allowed}. The main contribution comes from the
first term. To evaluate it, note that
when $t=0,$
\[\dell_{x_k}\dell_{y_j}|_{x=y}e^{\frac{i}{h}\lr{\widehat{S}(t,x,p)-y\cdot p
+ Et}}=
\delta_{kj}\lr{h^{-2}p_k\cdot
p_j}e^{\frac{i}{h}\lr{\widehat{S}(t,x,p)-y\cdot p + Et}}.\]
We therefore have
\[\int_{W_{x,E}}\dell_{x_k}\dell_{y_j}|_{x=y}e^{\frac{i}{h}\lr{\widehat{S}(t,x,p)-y\cdot
p + Et}}d\mu_{W_{x,E}} =
h^{-2}\delta_{jk}\frac{\lr{2E-\abs{x}^2}^{\frac{d}{2}}}{d}.\]
Applying Lemma \ref{L:SP} proves (\ref{E:Projector Two Deriv on Diag}) and
completes the proof of Proposition \ref{P:Cov Matrix
Asymptotics} in the allowed region.
\end{proof}

\end{document}